\documentclass[secnum,leqno]{rims-bessatsu}
\usepackage{amssymb, amsmath}
\usepackage{graphics}
\usepackage[dvips]{graphicx}
\usepackage{eepic}
\usepackage{epic}
\usepackage[all]{xy}
\usepackage{tabmac}
\usepackage[vcentermath]{youngtab}
\newtheorem{thm}{Theorem}[section]

\newtheorem{lmm}[thm]{Lemma}

\newtheorem{prp}[thm]{Proposition}

\theoremstyle{definition}

\newtheorem{exa}[thm]{Example}
\theoremstyle{remark}
\newtheorem*{rem}{Remark}
\newcommand\dia{\diamondsuit}
\newcommand{\geh}{\mathfrak{g}}
\newcommand\la{\lambda}

\newcommand\ol{\overline}
\newcommand\ot{\otimes}

\newcommand\wt{\mathrm{wt}\,}
\newcommand{\Z}{\mathbb{Z}}
\newcommand{\vdom}{{\Yvcentermath1\Yboxdim4pt \,\yng(1,1)\,}}
\newcommand{\hdom}{{\Yvcentermath1\Yboxdim4pt \,\yng(2)\,}}
\newcommand{\cell}{{\Yvcentermath1\Yboxdim4pt \,\yng(1)\,}}
\newcommand{\vdoms}{{\Yvcentermath1\Yboxdim3pt \,\yng(1,1)\,}}
\newcommand{\cells}{{\Yvcentermath1\Yboxdim3pt \,\yng(1)\,}}
\title{$X=K$ under Review}
\author{C\'edric \textsc{Lecouvey}\footnote{LMPT, Universit\'e Fran\c{c}ois Rabelais 
Tours.\newline e-mail: \texttt{cedric.lecouvey@lmpt.univ-tours.fr}},
Masato \textsc{Okado}\footnote{Department of Mathematical Science, Graduate School of Engineering Science, 
Osaka University, Toyonaka, Osaka 560-8531, Japan.\newline e-mail: \texttt{okado@sigmath.es.osaka-u.ac.jp}}
~and Mark \textsc{Shimozono}\footnote{Department of Mathematics, Virginia Tech,
Blacksburg, VA 24061-0123 USA. \endgraf e-mail: \texttt{mshimo@math.vt.edu}}}
\AuthorHead{C. Lecouvey, M. Okado and M. Shimozono}         
\support{Partially supported by ANR-09-JCJC-0102-01 (C.L.), 
Grants-in-Aid for Scientific Research No. 20540016 (M.O.),
NSF DMS-0652641 and DMS-0652648 (M.S.)}  
\VolumeNo{x}           
\YearNo{200x}           
\PagesNo{000--000}      
\communication{Received April 20, 200x.
Revised September 11, 200x.}      
\begin{document}
%

\maketitle

\begin{abstract}      
We review the $X=K$ conjecture and important ingredients for the proof.
We also attach notes on the rank estimate for the $X=K$ theorem to hold and on the strange relation
that was found to be valid without the assumption that the rank is sufficiently large.
Using the latter one obtains an algorithm to calculate the image of the combinatorial $R$-matrix and
the value of the coenergy function.
\end{abstract}

\section{Review on $X=K$}

Let $\geh$ be an affine algebra of nonexceptional type and $I=\{0,1,\ldots,n\}$ the index set of its 
Dynkin nodes.
Let $0\in I$ as specified in \cite{Kac} and set $I_0=I\setminus\{0\}$. For a pair $(r,s)$ ($r\in I_0,
s\in\Z_{>0}$) there exists a crystal $B^{r,s}$ called the Kirillov-Reshetikhin (KR) crystal \cite{FOS}.
It is a crystal base in the sense of Kashiwara \cite{Ka} of the Kirillov-Reshetikhin module $W^{r,s}(a)$
for a suitable parameter $a$ \cite{O,OSch} over the quantum affine algebra $U'_q(\geh)$ without the degree 
operator $q^d$. Let $B$ be a tensor product of KR crystals $B=B^{r_1,s_1}\ot B^{r_2,s_2}\ot\cdots\ot 
B^{r_L,s_L}$, and for a subset $J$ of $I$ set $\mathrm{hw}_{J}(B)=\{b\in B\mid e_ib=0\text{ for any }
i\in J\}$ where $e_i$ is the Kashiwara operator acting on $B$. We call an element of $\mathrm{hw}_J(B)$ 
$J$-highest. For an $I_0$-weight $\la$ we define the 1-dimensional sum $\ol{X}_{\la,B}(q)$ by 
\[
\ol{X}_{\la,B}(q)=\sum_{b\in\mathrm{hw}_{I_0}(B),\wt b=\la}q^{\ol{D}(b)}.
\]
Here $\ol{D}:B\rightarrow\Z$ is the intrinsic coenergy function (see \cite[\S3.5]{LOS}). Assume now 
that $n=|I_0|$ is sufficiently large. (We make an attempt to estimate $n$ such that our main theorem
holds.) Then it can be shown that $\ol{X}_{\la,B}(q)$ depends only on the attachment of the node $0$ 
to the rest of the Dynkin diagram of $\geh$. In the table below we list all possibilities of the 
attachment of $0$ and enumerate the corresponding nonexceptional affine algebras. 
\[
\begin{array}{|c|c|c|c|} \hline
\text{Dynkin} & \geh & \dia \\ \hline \hline
{\xymatrix@R=1ex{
{} & *{\circ}<3pt> \ar@{-}[r] & *{\circ}<3pt> \ar@{..}[r] & \\
*{\circ}<3pt> \ar@{-}[ur]^<{0} \ar@{-}[dr] & & & \\
{} & *{\circ}<3pt> \ar@{-}[r] & *{\circ}<3pt> \ar@{..}[r] & 
}}
&
\lower4mm\hbox{$A_n^{(1)}$}
&
\lower4mm\hbox{$\varnothing$}
\\
\hline
{\xymatrix@R=1ex{
*{\circ}<3pt> \ar@{-}[dr]^<{0} & {} & {} & \\
{} & *{\circ}<3pt> \ar@{-}[r]& *{\circ}<3pt> \ar@{..}[r] & \\
*{\circ}<3pt> \ar@{-}[ur] & {} & {} & 
}}
&
\lower4mm\hbox{$B_n^{(1)},D_n^{(1)},A_{2n-1}^{(2)}$}
&
\lower4mm\hbox{$\vdom$}
\\
\hline
{\xymatrix@R=1ex{
*{\circ}<3pt> \ar@{=>}[r]^<{0} & *{\circ}<3pt> \ar@{-}[r] & *{\circ}<3pt> \ar@{..}[r] & \\
}}
& C_n^{(1)} & \hdom
\\ \hline
{\xymatrix@R=1ex{
*{\circ}<3pt> \ar@{<=}[r]^<{0} & *{\circ}<3pt> \ar@{-}[r] & *{\circ}<3pt> \ar@{..}[r] & \\
}}
& A_{2n}^{(2)},D_{n+1}^{(2)} & \cell
\\ \hline
\end{array}
\]
Hence, we have four kinds of ``stable" 1-dimensional sums denoted by $\ol{X}^\dia_{\la,B}(q)$ ($\dia=
\varnothing,\vdom,\hdom,\cell$). Then the so-called $X=K$ conjecture proposed by Shimozono and Zabrocki
\cite{Sh,SZ} is stated as follows. 
\begin{thm}[\cite{LOS}] \label{th:main}
For $\dia\ne\varnothing$,
\[
\ol{X}^\dia_{\la,B}(q)=q^{\frac{|B|-|\la|}{|\dia|}}
\sum_{\mu\in\mathcal{P}^\dia_{|B|-|\la|},\nu\in\mathcal{P}^\cells_{|B|}}c_{\la\mu}^\nu
\ol{X}^\varnothing_{\nu,B}(q^\frac2{|\dia|}).
\]
Here $|B|=\sum_{j=1}^Lr_js_j,|\la|=\sum_i\la_i$ for $\la=(\la_1,\la_2,\ldots)$
where a non-spin weight $\la$ is identified with a partition by the standard way,
$\mathcal{P}_N^\dia=$set of partitions of $N$ tiled by $\dia$, and
$c_{\la\mu}^\nu$ stands for the Littlewood-Richardson coefficient.
\end{thm}

We sketch the proof of this theorem from \cite{LOS}. Since $\ol{X}^\dia_{\la,B}(q)$ depends only on 
the symbol $\dia$, we choose an affine algebra $\geh^\dia$ from each kind such that $i\mapsto n-i$ 
($i\in I$) gives a Dynkin diagram automorphism. Namely, we set $\geh^\dia=A_n^{(1)},D_n^{(1)},C_n^{(1)},
D_{n+1}^{(2)}$ for $\dia=\varnothing,\vdom,\hdom,\cell$. Let $\dia=\vdom,\hdom$ or $\cell$ from now on.
Then there exists an automorphism $\sigma$ on the KR crystal $B^{r,s}$ for $\geh^\dia$ satisfying
\[
\sigma(e_ib)=e_{n-i}\sigma(b)
\]
for any $i\in I,b\in B^{r,s}$. This automorphism $\sigma$ is extended to $B$ by 
$\sigma(b)=\sigma(b_1)\ot\sigma(b_2)\ot\cdots\ot\sigma(b_L)$. Then the important facts for the proof
are summarized as follows.
\begin{itemize}
\item[(i)] $\sigma$ restricts to the following bijection.
\[
\left\{
\begin{array}{c}
\hbox{$I_0$-highest elements}\\
\hbox{in $B$ of $\wt\la$}
\end{array}
\right\}\overset{\sigma}{\longrightarrow}
\left\{
\begin{array}{c}
\hbox{$I\setminus\{0,n\}$-highest elements}\\
\hbox{in $\max(B)$ of $\wt\ol{\la}$}
\end{array}
\right\}
\]
Here $\max(B)=\bigoplus_\gamma B(\gamma)$, where $B(\gamma)$ is the highest weight 
$U_q(\geh^\dia_{I_0})$-crystal of highest weight $\gamma$ and $\gamma$ runs over all weights with
$|\gamma|=|B|$ such that $B(\gamma)$ appears in the restriction of $B$. 
Namely, $\max(B)$ is the disjoint union of classical highest weight crystals of maximal highest weights. 
We remark that $\geh_{I\setminus\{0,n\}}$ is isomorphic to $A_{n-1}$ and set $\ol{\la}=(-\la_n,\ldots,-\la_1)$
if $\la=(\la_1,\ldots,\la_n)$.
\item[(ii)] $\ol{D}(b)=\ol{D}(\sigma(b))+(|B|-|\hbox{wt}\,b\,|)/|\dia|$ for $b\in\hbox{hw}_{I_0}(B)$.
\item[(iii)] We have $[\left.V^G(\nu)\right\downarrow^G_{GL_n}:V^{GL_n}(\bar{\la})]
	=\sum_{\mu\in\mathcal{P}^\dia}c_{\la\mu}^\nu$, where $G=SO_{2n},Sp_{2n},SO_{2n+1}$ for 
	$\dia=\vdom,\hdom,\cell$ and $V^G(\nu)$ stands for the irreducible $G$-module of non-spin highest 
	weight $\nu$.
\item[(iv)] If we represent elements of a KR crystal by Kashiwara-Nakashima tableaux \cite{KN}, $I_0$-highest 
	elements in $\max(B)$ contain no barred letters and can therefore be viewed as elements of type $A$.
	Under this correspondence we have $\ol{D}^\dia(b)=\frac{2}{|\dia|}\ol{D}^\varnothing(b)$.
\end{itemize}
Once these properties are established, our theorem can easily be proved as
\begin{align*}
\ol{X}^\dia_{\la,B}(q)&=\sum_{b\in\hbox{\scriptsize hw}_{I_0}(B),\wt b=\la}q^{\ol{D}(b)}\\
&\overset{(ii)}{=}q^d\sum_bq^{\ol{D}(\sigma(b))}\\
&\overset{(i)(iii)}{=}q^d\sum_{\mu\in\mathcal{P}^\dia,\nu\in\mathcal{P}}c_{\la\mu}^\nu
\sum_{\hat{b}\in\hbox{\scriptsize hw}_{I_0}(\max(B)),\wt \hat{b}=\nu}q^{\ol{D}(\hat{b})}\\
&\overset{(iv)}{=}q^d\sum_{\mu,\nu}c_{\la\mu}^\nu\ol{X}^\varnothing_{\nu,B}(q^\frac{2}{|\dia|})
\end{align*}
where we have set $d=(|B|-|\la|)/|\dia|$.

\begin{exa}
Consider the affine algebra $\geh=D_6^{(1)}$ of kind $\vdom$ and the following three elements of 
$B=B^{2,2}\ot B^{3,1}\ot B^{1,3}$. They all have weight $\la=(211)$. Their images by the automorphism 
$\sigma$ are also given.
\[
\begin{array}{|c|c|} \hline
b & \sigma(b) \\ \hline\vbox{\vspace{9mm}}
\quad\vcenter{
\tableau[sby]{2\\1}
}
\ot
\vcenter{
\tableau[sby]{\ol{3}\\3\\1}
}
\ot
\vcenter{
\tableau[sby]{1&3&\ol{1}}
}\quad&\quad
\vcenter{
\tableau[sby]{6&\ol{5}\\ \ol{6}&\ol{6}}
}
\ot
\vcenter{
\tableau[sby]{\ol{5}\\\ol{6}\\5}
}
\ot
\vcenter{
\tableau[sby]{5&\ol{5}&\ol{4}}
}\quad\\[6mm]
\quad\vcenter{
\tableau[sby]{2\\1}
}
\ot
\vcenter{
\tableau[sby]{\ol{4}\\4\\3}
}
\ot
\vcenter{
\tableau[sby]{1&3&\ol{3}}
}\quad&\quad
\vcenter{
\tableau[sby]{6&\ol{5}\\ \ol{6}&\ol{6}}
}
\ot
\vcenter{
\tableau[sby]{\ol{4}\\ \ol{5}\\5}
}
\ot
\vcenter{
\tableau[sby]{3&\ol{6}&\ol{3}}
}\quad\\[6mm]
\quad\vcenter{
\tableau[sby]{2&2\\1&1}
}
\ot
\vcenter{
\tableau[sby]{\ol{2}}
}
\ot
\vcenter{
\tableau[sby]{1&3&\ol{1}}
}\quad&\quad
\vcenter{
\tableau[sby]{\ol{5}&\ol{5}\\ \ol{6}&\ol{6}}
}
\ot
\vcenter{
\tableau[sby]{6\\ \ol{6}\\5}
}
\ot
\vcenter{
\tableau[sby]{5&\ol{5}&\ol{4}}
}\quad\\[6mm]\hline
\end{array}
\]
By the property (i) each $\sigma(b)$ should belong to $\max(B)$. Actually, by applying the raising 
operators $e_i$ ($i\in I_0$) one finds that these three elements have the common $I_0$-highest element
\[
\hat{b}=\vcenter{
\tableau[sby]{2&2\\1&1}
}
\ot
\vcenter{
\tableau[sby]{4\\3\\1}
}
\ot
\vcenter{
\tableau[sby]{2&3&5}
}
\]
of weight $\nu=(33211)$.
It is also true that they are the all $I_0$-highest elements in $B$ whose images under $\sigma$ belong
to the same $I_0$-component as the above one. We can check the property (iii), since we get $c_{\la\mu}^\nu
=1$ if $\mu=(33)$, $=2$ if $\mu=(2211)$, $=0$ if $\mu$ are other elements in $\mathcal{P}^\vdom$ and therefore
\[
\sum_{\mu\in\mathcal{P}^\vdoms}c_{\la\mu}^\nu=3.
\]
The intrinsic coenergies $\ol{D}(b)$ are all equal and can be calculated using the property (ii) as
\[
\ol{D}(b)=\ol{D}(\sigma(b))+\frac{10-4}2.
\]
Since $\ol{D}$ is constant on each $I_0$-component, we have $\ol{D}(\sigma(b))=\ol{D}(\hat{b})$. The r.h.s.
is calculated to be $4$ using the knowledge of the type $A$ crystal \cite{Sh}. Therefore we obtain 
$\ol{D}(b)=7$.
\end{exa}

\begin{rem}
The so-called $X=M$ conjecture \cite{HKOTY,HKOTT} claims that the 1-dimensional sum $\ol{X}_{\la,B}(q)$
is equal to the fermionic formula $\ol{M}(\la,\mathbf{L};q)$. Hence, when $n$ is sufficiently large, one
can expect that $\ol{M}(\la,\mathbf{L};q)$ has a similar formula to Theorem \ref{th:main}. This is 
confirmed in \cite{OSa}. Combining these results with \cite{KSS}, the $X=M$ conjecture is settled 
when the affine algebra is of nonexceptional type and its rank is sufficiently large.
\end{rem}

\section{Rank Estimate}

In this section we make an attempt to estimate $n$ such that Theorem \ref{th:main} holds.
\begin{prp}
Let $\ell$ be the length of $\la$. Then
Theorem \ref{th:main} holds if 
\[
n>(2\ell+1)+|B|-|\la|.
\]
\end{prp}
\begin{proof}
The obstacle for the theorem to hold lies in the fact that the property (i) is no longer valid when $n$ 
is not large enough. In \cite{LOS} this property is stated as Theorem 7.1. In view of the proof there one
recognizes that if $n$ is so large that $\sigma(b)$ for $b\in\mathrm{hw}_{I_0}(B)$ is contained in 
$\max(B)$, then everything is ok. Using row and box splittings in \cite[\S6]{LOS} one can also reduce
the proof when $B$ is a tensor product of the simplest KR crystal $B^{1,1}$, that is, $B=(B^{1,1})^{\ot L}$.
Hence, our task is to estimate $n$ such that $\sigma(b)$ belongs to $\max((B^{1,1})^{\ot L})$ for any 
$b\in\mathrm{hw}_{I_0}((B^{1,1})^{\ot L})$ of weight $\la$.

Recall that an element of $(B^{1,1})^{\ot L}$ can be regarded as a word of length $L$ from the alphabet
\[
\{(\phi,)1,2,\ldots,n,(0,)\ol{n},\ol{n-1},\ldots,\ol{1}\}.
\]
Here letters in parentheses are only for $\dia=\cell$. Let $b$ be a word of length $L$ that is $I_0$-highest.
Then the letters $b$ lie in the set $\{(\phi,)1,2,\ldots,m,\ol{m},\ol{m-1},\ldots,\ol{1}\}$ for some 
$m(\ge\ell)$. Let $c_z$ be the number of letters $z$ in $b$. Then we have $c_j-c_{\ol{j}}=\la_j>0$ for
$1\le j\le \ell$ and $c_{j}=c_{\ol{j}}>0$ for $\ell<j\le m$. Since $\sum_{j=1}^m(c_j+c_{\ol{j}})\le L$,
we have $\sum_{j=1}^mc_{\ol{j}}\le\frac{L-|\la|}2$. Setting $M=\max_{\ell<j\le m}c_{\ol{j}}$, we get
\begin{equation} \label{eq1}
M+(m-\ell-1)\le\frac{L-|\la|}2.
\end{equation}

Next recall the insertion algorithms from \cite{L}. For a word or element of a tensor product of $B^{1,1}$
the insertion algorithm tells us the highest weight of the $I_0$-component the word belongs to. In our case
we wish to apply this algorithm to $\sigma(b)$ to see if the shape of the resulting tableau has $L$ nodes.
This is equivalent to say that at each step of insertion of a letter to a column the resulting column remains
to be admissible. This in particular means that if letter $x$ and $\ol{x}$ coexist at position $p$ and $q$
in some column of height $N$, then we have
\begin{equation} \label{eq2}
x\ge p+(N+1-q). 
\end{equation}
Let us obtain the minimal possible unbarred letter $X$
that could appear in the course of insertion algorithms. Note that letters of $\sigma(b)$ lie in 
$\{n-m+1,n-m+2,\ldots,n,(0,)\ol{n},\ol{n-1},$ 
$\ldots,\ol{n-m+1}\}$. Since plactic relations of \cite{L} contain
$x\ol{x}y\equiv(\ol{x-1})(x-1)y$, a pair $(n-m+1,\ol{n-m+1})$ could create $(n-m+1-M,\ol{n-m+1-M})$.
Hence we can set $X=n-m+1-M$. The worst situation that could break \eqref{eq2} is that there exist
pairs $(X+j-1,\ol{X+j-1})$ for any $1\le j\le M+m$ in the first column during the insertion procedure. 
The condition for such a column to be admissible is given by
\begin{equation} \label{eq3}
n\ge2(M+m).
\end{equation}
In view of \eqref{eq1} we obtain the desired result.
\end{proof}

\section{Strange Relation}

In this section we show the following proposition and apply it to give an algorithm to obtain 
the image of the combinatorial $R$-matrices and the value of the coenergy function $\ol{H}$.
As we see in the proof, we do not assume the rank is sufficiently large. So the algorithm can
be used for any $n$. However, we need to restrict our affine algebras to $\geh^\dia$ ($\dia=\vdom,
\hdom,\cell$), since we use the automorphism $\sigma$. From the same reason we exclude the KR
crystals $B^{n-1,s}$ and $B^{n,s}$ for $\geh^\vdom=D_n^{(1)}$ and $n$ is odd.
\begin{prp} \label{prop:strange rel}
Let $B$ be a tensor product of KR crystals. Suppose $b\in\mathrm{hw}_{I_0}(B)$. Then we have
\[
\ol{D}(b)-\ol{D}(\sigma(b))=\frac{|B|-|\la(b)|}{|\dia|}.
\]
Here $\la(b)$ stands for the partition corresponding to the weight of $b$.
\end{prp}
This proposition is essentially the same as Theorem 8.1 of \cite{LOS} except that we do not assume
$n$ is sufficiently large. We prepare a lemma. Let us extend the definition of $\la(b)$ to an arbitrary
element $b$ by $\la(b)=(\la_1,\la_2,\ldots,\la_n)$ where $\la_i=(\wt b,\epsilon_i)$ and 
$\{\epsilon_i\}_{1\le i\le n}$ stands for the standard basis vectors of the weight lattice.
We note that $\la(b)$ is not necessarily a partition. Some $\la_i$'s may be negative. 
Hence $|\la|=\sum_i\la_i$ may also become negative.
\begin{lmm} Let $B_1,B_2$ be single KR crystals. Let $b_1\ot b_2$ be an element of $B_1\ot B_2$
and suppose it is mapped to $b'_2\ot b'_1$ by the affine crystal isomorphism. Then we have
\begin{equation} \label{strange H}
\ol{H}(b_1\ot b_2)-\ol{H}(\sigma(b_1)\ot\sigma(b_2))=\frac{|\la(b'_2)|-|\la(b_2)|}{|\dia|}.
\end{equation}
\end{lmm}

\begin{proof}
Since $B_1\ot B_2$ is connected, it is sufficient to show 
\begin{itemize}
\item[(i)] if $b_1=u(B_1),b_2=u(B_2)$ (see \cite[\S3.4]{LOS} for the definition of $u(B_i)$), 
	\eqref{strange H} holds, and
\item[(ii)] \eqref{strange H} with $b_1\ot b_2$ replaced by $e_i(b_1\ot b_2)$ holds,
	provided that \eqref{strange H} holds and $e_i(b_1\ot b_2)\ne0$.
\end{itemize}
For (i) recall $b'_1=b_1,b'_2=b_2$ if $b_1=u(B_1),b_2=u(B_2)$. Since $u(B_1)\ot u(B_2)$
can be reached from $\sigma(u(B_1))\ot\sigma(u(B_2))$ by applying $e_i$ ($i\ne0$), we have
$\ol{H}(u(B_1)\ot u(B_2))=\ol{H}(\sigma(u(B_1))\ot\sigma(u(B_2)))=0$. Hence (i) is verified.

For (ii) recall $|\la(e_ib)|-|\la(b)|=-|\dia|\,(i=0),=|\dia|\,(i=n),=0\,(\text{otherwise})$.
If $i\ne0,n$, both sides do not change when we replace $b_1\ot b_2$ with $e_i(b_1\ot b_2)$.
If $i=0$, the first term of the l.h.s decreases by one in case LL, increases by one in case RR,
and does not change in case LR or RL. (For the meaning of LL, etc, see \cite[Prop. 3.7]{LOS}.)
The second term does not change, while the r.h.s varies in the same way as the first term of the l.h.s.
The $i=n$ case is similar. 
\end{proof}

\begin{proof}[Proof of Proposition \ref{prop:strange rel}]
Let $B=B^{r_1,s_1}\ot\cdots\ot B^{r_p,s_p}$. We prove by induction on $p$. When $p=1$,
the proof is the same as in \cite[Th. 8.1]{LOS}.

Let $B=B'\ot B^{r_p,s_p}$ and $b_1\ot b_2\in B'\ot B^{r_p,s_p}$ is mapped to $b'_2\ot b'_1
\in B^{r_p,s_p}\ot B'$ by the affine crystal isomorphism. Then $\sigma(b_1)\ot\sigma(b_2)$
should be mapped to $\sigma(b'_2)\ot\sigma(b'_1)$. Using (3.52) of \cite{LOS} we have
\begin{align*}
\ol{D}(b)&=\ol{D}(b_1)+\ol{D}(b'_2)+\ol{H}(b_1\ot b_2),\\
\ol{D}(\sigma(b))&=\ol{D}(\sigma(b_1))+\ol{D}(\sigma(b'_2))+\ol{H}(\sigma(b_1)\ot\sigma(b_2)).
\end{align*}
On the other hand, by the previous lemma and \cite[Lemma 5.2]{Memoir} we have
\[
\ol{H}(b_1\ot b_2)-\ol{H}(\sigma(b_1)\ot\sigma(b_2))=\frac{|\la(b'_2)|-|\la(b_2)|}{|\dia|}.
\]
Using the induction hypothesis we obtain
\begin{align*}
\ol{D}(b)-\ol{D}(\sigma(b))&=\frac{|B'|-|\la(b_1)|}{|\dia|}+\frac{|B^{r_p,s_p}|-|\la(b'_2)|}{|\dia|}
+\frac{|\la(b'_2)|-|\la(b_2)|}{|\dia|}\\
&=\frac{|B|-|\la(b)|}{|\dia|}
\end{align*}
as desired.
\end{proof}

Using Proposition \ref{prop:strange rel} we can give an algorithm to obtain 
the image of the combinatorial $R$-matrix and the value of the coenergy function $\ol{H}$.
This algorithm turns out effective when it is calculated using computer. For the calculation
of $\sigma$ see \cite[Appendix B.2]{LOS}.
Let $B_i=B^{r_i,s_i}$ ($i=1,2$) be KR crystals. The affine crystal isomorphism 
\[
R: B_1\ot B_2\longrightarrow B_2\ot B_1,
\]
which is known to exist uniquely, is called the combinatorial $R$-matrix. For an element $b_1\ot b_2
\in B_1\ot B_2$ we wish to calculate the image $R(b_1\ot b_2)$. Since the application of Kashiwara
operators $e_i,f_i$ for $i\ne0$ is not difficult, one can reduce its calculation to $I_0$-highest 
elements of $B_1\ot B_2$. For an element $b$ in an $I_0$-component let $High(b)$ stand for the $I_0$-highest 
element and set $\Phi=High\circ\sigma$. From Proposition \ref{prop:strange rel} and the invariance of
$\ol{D}$ by classical Kashiwara operators, one has
\begin{equation} \label{eq4}
\ol{D}(\Phi(b_1\ot b_2))=\ol{D}(b_1\ot b_2)-\frac{|B_1\ot B_2|-|\la(b_1\ot b_2)|}{|\dia|}.
\end{equation}
Note that the second term of the above relation vanishes, if and only if $b_1\ot b_2\in\max(B_1\ot B_2)$.
Since the application of $\Phi$ decreases $\ol{D}$ and $\ol{D}$ takes a finite number of values, 
there exists a positive integer $m$ such that 
$\Phi^m(b_1\ot b_2)\in\max(B_1\ot B_2)$. Namely, there exist sequences $\mathbf{a}_1,\ldots,\mathbf{a}_m$
from $I_0$ and an element $\hat{b}_1\ot\hat{b}_2\in\max(B_1\ot B_2)$ such that
\[
\hat{b}_1\ot\hat{b}_2=
(e_{\mathbf{a}_m}\circ\sigma\circ\cdots\circ e_{\mathbf{a}_1}\circ\sigma)(b_1\ot b_2),
\]
or equivalently,
\[
b_1\ot b_2=
(\sigma\circ f_{\mathrm{Rev}(\mathbf{a}_1)}\circ\cdots\circ\sigma\circ f_{\mathrm{Rev}(\mathbf{a}_m)})
(\hat{b}_1\ot\hat{b}_2).
\]
Here for $\mathbf{a}=(i_1,\ldots,i_l)$ $e_\mathbf{a}$ stands for $e_{i_1}\cdots e_{i_l}$ 
($f_\mathbf{a}$ is similar) and $\mathrm{Rev}(\mathbf{a})=(i_l,\ldots,i_1)$. Since $R$ commutes with
$e_\mathbf{a},f_\mathbf{a}$ and $\sigma$, we have
\[
R(b_1\ot b_2)=
(\sigma\circ f_{\mathrm{Rev}(\mathbf{a}_1)}\circ\cdots\circ\sigma\circ f_{\mathrm{Rev}(\mathbf{a}_m)})
R(\hat{b}_1\ot\hat{b}_2).
\]
On the other hand, for an $I_0$-highest element in $\max(B_1\ot B_2)$ the image of $R$ is easily 
calculated (see \cite[\S9.1]{LOS}). Hence, one can calculate $R(b_1\ot b_2)$.

We proceed to the calculation of $\ol{H}(b_1\ot b_2)$. Firstly, one has the relation
\begin{equation} \label{eq5}
\ol{D}(b_1\ot b_2)=\ol{D}(b_1)+\ol{D}(b'_2)+\ol{H}(b_1\ot b_2),
\end{equation}
where $R(b_1\ot b_2)=b'_2\ot b'_1$.
The l.h.s is has been obtained in the course of the previous process and the known result
of the value of $\ol{D}$ for an element in $\max(B_1\ot B_2)$. For $I_0$-highest elements
$b_1,b'_2$ of a single KR crystal the value of $\ol{D}$ is calculated as
\[
\ol{D}(b)=\frac{rs-|\la(b)|}{|\dia|}\quad\text{for }b\in B^{r,s}.
\]
Therefore, one obtains $\ol{H}(b_1\ot b_2)$.

\begin{exa}
Consider the affine algebra $\geh=D_6^{(1)}$ of kind $\vdom$ and the following element of
$B^{4,3}\ot B^{3,3}$.
\[
b_1\ot b_2=
\vcenter{
\tableau[sby]{4\\ 3\\ 2&2\\ 1&1}
}
\ot
\vcenter{
\tableau[sby]{\ol{2}\\ \ol{3}\\ 3&4&\ol{4}}
}
\]
Then $\Phi(b_1\ot b_2)$ and $\Phi^2(b_1\ot b_2)$ are given as follows.
\[
\Phi(b_1\ot b_2)=
\vcenter{
\tableau[sby]{4&4&4\\ 3&3&3\\ 2&2&2\\ 1&1&1}
}
\ot
\vcenter{
\tableau[sby]{3&5&\ol{6}\\ 2&2&6\\ 1&1&5}
}\;,\qquad
\Phi^2(b_1\ot b_2)=
\vcenter{
\tableau[sby]{4&4&4\\ 3&3&3\\ 2&2&2\\ 1&1&1}
}
\ot
\vcenter{
\tableau[sby]{3&5&6\\ 2&2&5\\ 1&1&1}
}
\]
with 
\begin{align*}
\mathbf{a}_1=
(&64354643215432643215432643564321543264354643215432643546643215432643546),\\
\mathbf{a}_2=
(&66456435464325436643215432646432154326435643215432643546432154326435466\\
&43215432643546643215432643546).
\end{align*}
Since one knows the image of $R$ of $\Phi^2(b_1\ot b_2)$ is given by
\[
\vcenter{
\tableau[sby]{3&3&3\\ 2&2&2\\ 1&1&1}
}
\ot
\vcenter{
\tableau[sby]{4&5&6\\ 3&4&5\\ 2&2&4\\ 1&1&1}
}\;,
\]
one obtains
\[
R(b_1\ot b_2)=b'_2\ot b'_1=
\vcenter{
\tableau[sby]{3&3\\ 2&2\\ 1&1&1}
}
\ot
\vcenter{
\tableau[sby]{\ol{2}\\ \ol{3}\\ \ol{4}&\ol{1}\\ 4&4}
}\;.
\]

We proceed to the calculation of $\ol{H}(b_1\ot b_2)$. By using \eqref{eq4} twice and
$\ol{D}(\Phi^2(b_1\ot b_2))=3$ one gets $\ol{D}(b_1\ot b_2)=12$. Since $\ol{D}(b_1)=3,\ol{D}(b'_2)=1$,
we obtain $\ol{H}(b_1\ot b_2)=8$ from \eqref{eq5}.
\end{exa}



\begin{thebibliography}{99}
%
%

\bibitem{FOS}
G.~Fourier, M.~Okado and A.~Schilling,
\textit{Kirillov-Reshetikhin crystals for nonexceptional types},
Adv. in Math. {\bf 222} (2009) 1080--1116.

\bibitem{HKOTT}
G.~Hatayama, A.~Kuniba, M.~Okado, T.~Takagi and Z.~Tsuboi,
\textit{Paths, crystals and fermionic formulae}, MathPhys Odyssey 2001, 205--272, 
Prog.\ Math.\ Phys.\ {\bf 23}, Birkh\"auser Boston, Boston, MA, 2002.

\bibitem{HKOTY}
G.~Hatayama, A.~Kuniba, M.~Okado, T.~Takagi and Y.~Yamada, 
\textit{Remarks on fermionic formula},
Contemporary Math.\ {\bf 248} (1999) 243--291.

\bibitem{Kac}
V.~G.~ Kac, 
\textit{``Infinite Dimensional Lie Algebras,"}
3rd ed., Cambridge Univ. Press, Cambridge, UK, 1990.

\bibitem{Ka}
M.~Kashiwara, 
\textit{On crystal bases of the $q$-analogue of universal enveloping algebras},
Duke Math. J. {\bf 63} (1991) 465--516.

\bibitem{KN}
M.~Kashiwara and T.~Nakashima,
\textit{Crystal graphs for representations of the $q$-analogue of classical Lie algebras}, 
J. Algebra \textbf{165} (1994), no. 2, 295--345.

\bibitem{KSS}
A.~N.~Kirillov, A.~Schilling and M.~Shimozono,
\textit{A bijection between Littlewood-Richardson tableaux and rigged configurations}, 
Selecta Math. (N.S.) {\bf 8} (2002) 67--135.

\bibitem{L}
C.~Lecouvey, 
\textit{Schensted-type correspondence, plactic monoid and jeu de taquin for type $C_n$},
J. Algebra \textbf{247} (2002) 295--331;
\textit{Schensted-type correspondences and plactic monoids for types $B_n$ and $D_n$}, 
J. Algebraic Combin. \textbf{18} (2003) 99--133.

\bibitem{LOS}
C.~Lecouvey, M.~Okado and M.~Shimozono,
\textit{Affine crystals, one-dimensional sums and parabolic Lusztig $q$-analogues},
arXiv:1002.3715.

\bibitem{Memoir}
M.~Okado,
\textit{$X=M$ conjecture}, MSJ Memoirs \textbf{17} (2007) 43--73.

\bibitem{O}
M.~Okado,
\textit{Existence of crystal bases for Kirillov-Reshetikhin modules of type $D$},
Publ. RIMS \textbf{43} (2007) 977--1004.

\bibitem{OSa}
M.~Okado and R.~Sakamoto,
\textit{Stable rigged configurations for quantum affine algebras of nonexceptional types},
arXiv:1008.0460. 

\bibitem{OSch}
M.~Okado and A.~Schilling,
\textit{Existence of Kirillov-Reshetikhin crystals for nonexceptional types},
Representation Theory \textbf{12} (2008) 186--207.

\bibitem{Sh}
M.~Shimozono,
\textit{On the $X=M=K$ conjecture},
arXiv:math.CO/0501353.

\bibitem{SZ}
M.~Shimozono and M.~Zabrocki,
\textit{Deformed universal characters for classical and affine algebras},
J. of Algebra \textbf{299} (2006) 33--61.

\end{thebibliography}
\end{document}